\documentclass[12pt]{amsart}

\usepackage{color}

\usepackage{amsmath,amssymb,setspace,nicefrac}
\usepackage[active]{srcltx}
\usepackage[colorlinks, linkcolor=red, citecolor=blue, urlcolor=blue, hypertexnames=true]{hyperref}
\usepackage{amsrefs}

\allowdisplaybreaks
\usepackage[all]{xy}

\newcommand{\Z}{{\mathbb Z}}
\newcommand{\N}{{\mathbb N}}

\newcommand{\cO}{{\mathcal O}}

\newcommand{\cU}{{\mathcal U}}

\newcommand{\cG}{{\mathcal G}}

\newcommand{\Per}{{\rm Per}}

\newtheorem{thm}{Theorem}[section]
\newtheorem{cor}[thm]{Corollary}
\newtheorem{lem}[thm]{Lemma}

\newtheorem{definition}[thm]{Definition}
\newtheorem{example}[thm]{Example}
\newtheorem{remark}[thm]{Remark}

\theoremstyle{definition}

\begin{document}
\title[Continuous Cocycle Superrigidity for Shifts and Groups with One End]{Continuous Cocycle Superrigidity for Shifts and Groups with One End}

\author{Nhan-Phu Chung}
\address{Nhan-Phu Chung, Department of Mathematics, Sungkyunkwan University, Suwon 440-746, Korea.} 
\email{phuchung@skku.edu;phuchung82@gmail.com}
\author{Yongle Jiang}
\address{Yongle Jiang, Department of Mathematics, SUNY at Buffalo, NY 14260-2900, U.S.A.}
\email{yongleji@buffalo.edu}
\date{\today}
\maketitle
\begin{abstract}
In this article, we prove that if a finitely generated group $G$ is not torsion then a necessary and sufficient condition for every full shift over $G$ has (continuous) cocycle superrigidity is that $G$ has one end. It is a topological version of the well known Popa's measurable cocycle superrigidity theorem \cite{Popa2}. For the proof of sufficient condition, we introduce a new specification property for shifts over general groups which play a similar role as malleable property in the measurable setting. This new specification property is good enough for us to extend the method using homoclinic equivalence relation that was introduced by Klaus Schmidt to study cocycle rigidity for $\Z^d$-shifts \cite{Schmidt95}. Indeed, in this direction we prove this superrigidity result for more certain general systems. And for the converse, we apply Specker's characterization for ends of groups via the associated first cohomology groups to get the result. Finally, combining our results with \cite[Theorem 1.6]{Lixin}, we have an application in continuous orbit equivalence rigidity.
\end{abstract}
\section{Introduction}

Since its introduction \cite{Dye,Singer}, measurable orbit equivalence has received a tremendous attention not only from ergodic theorist but also analysts working on operator algebras. It is well-known that in the measurable setting (see e.g.\cite[Section 3.3]{Furman2}, \cite[Theorem 5.6]{Popa2}, \cite[Sections 4.2.9 and 4.2.11]{Zimmer}) it is a general principle that, roughly speaking, cocycle superrigidity implies orbit equivalence superrigidity.  Recently, using this principle and deformation/rigidity techniques in von Neumann algebras, many rigidity results for measurable orbit equivalence were established for certain actions of a large class of non-amenable groups \cite{Ioana, Popa2, Popa2008}. One of the breakthrough papers in this direction is Popa's measurable orbit equivalence superrigidity results which were obtained via his celebrated measurable cocycle superrigidity theorem \cite{Popa1,Popa2, Popa2008}.

A subgroup $G_0$ of $G$ is wq-normal if there exists no subgroup $G_0\leq H \lvertneqq G$ with $gHg^{-1}\cap H$ is finite for all $g\in G\setminus H$.

\textbf{Popa's measurable cocycle superrigidity theorem} \cite{Popa2,Popa2008} (for full shifts): Let $G$ be a group that contains an infinite wq-normal subgroup $G_0$ such that  the pair $(G,G_0)$ has relative property $(T)$, or such that $G_0$ is the direct product of an infinite group and a non-amenable group. Then the full shift action $G\curvearrowright\prod_{G}(A,\mu)$ is $\cU_{fin}$-measurable cocycle superrigid, for every finite set $A$.

Here $\cU_{fin}$ is the class of Polish groups which arise as closed subgroups of the unitary groups of finite von Neumann algebras. The class $\cU_{fin}$ contains compact Polish groups and countable groups.

On the other side, the interplay between topological orbit equivalence and $C^*$-algebras has been investigated extensively in the last 20 years. Although some fundamental progress in continuous orbit equivalence has been made just for some special cases, it already has had many applications in groupoids, homology theory, classification of C*-algebras, and K-theory \cite{BCW,BH1,BT,GPS2008,GPS2010,GPS,JM,Lixin,MM2014a,MM2014b,
 Matui2012,Matui2013,Matui2015,MST}. Similar to the measurable setting, continuous cocycles of a continuous group action are a fundamental tool for understanding the action. And it is natural to ask whether there is a continuous version of Popa's measurable cocycle superrigidity theorem and how one can apply it to study rigidity of continuous orbit equivalence.
 
We denote by $\mathcal{G}$ the class of finitely generated infinite discrete group $G$ such that the full shift $(G, A^G)$ is continuous $H$-cocycle rigid for every finite set $A$ and every countable discrete group $H$, i.e. every continuous cocycle $c: G\times A^G\to H$ is continuously cohomologous to a group homomorphism $\phi: G\to H$. Which group $G$ belongs to $\mathcal{G}$?
 
In \cite{Kam1, Kam2}, during studying classification of the isometric extensions of $\Z^d$- Bernoulli shifts, J. W. Kammeyer implicitly proved that when $d\geq 2$, every continuous cocycle for the shift action of $\Z^d$ on the full $d$-dimensional $k$-shift with values in $\Z/2\Z$ is continuously cohomologous to a homomorphism from $\Z^d$ to $\Z/2\Z$, i.e. the full shift $(\Z^d, A^{\Z^d})$ is continuous $\Z/2\Z$-cocycle rigid for all finite set $A$. Later, K. Schmidt proved that when $d\geq 2$, the full shift $(\Z^d, A^{\Z^d})$ is continuous $H$-cocycle rigid for every finite set $A$ and every discrete group $H$; hence $\Z^d$ belongs to $\mathcal{G}$ when $d\geq 2$ \cite{Schmidt95}. As far as we know, these were the only known groups in $\mathcal{G}$. 

Last year, Xin Li proved that the full shift $(G, A^G)$ is continuous $H$-cocycle rigid when the acting group $G$ is a torsion-free, duality group with cohomological dimension not equal to one and the target group $H$ is solvable \cite[Theorem 1.7]{Lixin}.

Inspired by the work of Li \cite{Lixin}, Popa \cite{Popa2,Popa2008}, and Schmidt \cite{Schmidt95}, in this article we give a necessary and sufficient condition for a non-torsion group $G$, i.e. $G$ contains at least one element with infinite order, to be a member in $\cG$ via the ends of the group, an important notion in geometric group theory. 
  
Now, we state our main results. From now on the acting group $G$ is always finitely generated. First, we establish continuous cocycle rigidity result for certain shifts of finite type.  
\begin{thm}\label{T-one end is sufficient}
Let $A$ be a finite set and $G$ a non-torsion group. Let $\sigma$ be the left shift action of $G$ on $A^G$ and $X\subset A^G$ be a subshift. Assume that $\sigma$ is topological mixing, and the homoclinic equivalence relation $\Delta_X$ has strong $a$-specification, for every element $a$ in $G$ with infinite order. If $G$ has one end then every continuous cocycle $c:G\times X\to H$ on every countable group $H$ is cohomologous to a homomorphism.  
\end{thm}
Second, we illustrate that for a finitely generated group $G$ (not necessary non-torsion), one end is a necessary condition for $G$ belonging to $\cG$. 

\begin{thm}\label{one end is necessary}
If $G$ has more than one end, then $(G, (\frac{\Z}{2\Z})^G)$ is not continuous $\frac{\Z}{2\Z}$-cocycle rigid.
\end{thm}
As a consequence, we get
\begin{cor}\label{main theorem}
If $G$ is (finitely generated and) not a torsion group, then $G\in\mathcal{G}$ if and only if $G$ has one end.
\end{cor}
We recall that motivated by his measurable cocycle superrigidity theorem, Popa asked the question whether $G$ satisfies measurable cocycle superrigidity theorem for Bernoulli shift actions if and only if (the nonamenable group) $G$ has zero first $\ell^2$-Betti number \cite[p. 251]{petersonsinclair}. For results on this question, we refer the reader to \cite{petersonsinclair}. Note that when a nonamenable group has zero first $\ell^2$-Betti number, then it has one end (see Example \ref{E-groups with one end}). Thus, our result might be considered as a solution to the counterpart of Popa's question in the continuous setting.  

Now let us describe briefly our main ideas in the proof of Theorem \ref{T-one end is sufficient}. 
Suppose we are given a continuous cocycle $c: G\times X\to H$, where $X\subset A^G$ is a subshift for some finite set $A$, and we want to show 
there exist a continuous map $b: X\to H$ and a group homomorphism $\phi: G\to H$ such that 
\begin{align}
\label{E-1}
c(g, x)=b(gx)\phi(g)b(x)^{-1},
\end{align} for all $g\in G$ and all $x\in X$. The key problem is to construct the ``mysterious" transfer map (also called untwister) $b$. Similar to Popa's approach in \cite{Popa2} for the measurable setting, we need to start at good element ``g'' in $G$ and find solutions $b$ of the equation (\ref{E-1}) in this direction $g$. After that, with certain nice property of the action on the direction $g$ we will get a universal solution $b$ for the whole group $G$. Specification property and one-ended groups have a such nice property as malleable actions and property (T) groups did in the superrigidity for the measurable setting.

In the first step, we start with some non-torsion element $g$ in $G$ and all $x\in X$. If $\phi(g)=b(gx)^{-1}c(g, x)b(x)$ for all $x$ in $X$, then $b(gx)^{-1}c(g, x)b(x)=b(gx')^{-1}c(g, x')b(x')$ for all $(x, x')\in X\times X$. 

Denote by $f(g^kx)=c(g, g^kx)$ for all $k\in\Z$, then the above equality implies $b(g^{k+1}x)b(g^{k+1}x')^{-1}=f(g^kx)b(g^kx)b(g^kx')^{-1}f(g^kx')^{-1}$ for all $(x, x')\in X\times X$ and all $k\in\Z$, which can be viewed as a recurrence relation, then repeating substitution, we get 
\begin{eqnarray*}
b(x)b(x')^{-1}&=&(\prod_{j=0}^kf(g^jx)^{-1})b(g^{k+1}x)b(g^{k+1}x')^{-1}(\prod_{j=0}^kf(g^jx')^{-1})^{-1},\\ 
b(x)b(x')^{-1}&=&(\prod_{j=1}^kf(g^{-j}x))b(g^{-k}x)b(g^{-k}x')^{-1}(\prod_{j=1}^kf(g^{-j}x'))^{-1} 
\end{eqnarray*}
for all $k\geq 0$ and all $(x, x')\in X\times X$.   

Then, the continuity of $b$ implies that when $x, x'$ have the same ``tails" (i.e. they lie in the homoclinic equivalence relation) , then $b(g^{k+1}x)b(g^{k+1}x')^{-1}=e=b(g^{-k}x)b(g^{-k}x')^{-1}$ for large enough $k$. So the above equality implies for large enough $k$, we must have this necessary condition $$b(x)b(x')^{-1}=(\prod_{j=0}^kf(g^jx)^{-1})(\prod_{j=0}^kf(g^jx')^{-1})^{-1}=(\prod_{j=1}^kf(g^{-j}x))(\prod_{j=1}^kf(g^{-j}x'))^{-1}.$$ 
If we further fix $x'$, then, $b(x)$ is already determined by these relations; however, in general, there is no reason to expect 
for large enough $k$, 
\begin{align}
\label{E-2}
(\prod_{j=0}^kf(g^jx)^{-1})(\prod_{j=0}^kf(g^jx')^{-1})^{-1}=(\prod_{j=1}^kf(g^{-j}x))(\prod_{j=1}^kf(g^{-j}x'))^{-1},
\end{align}
so the difficulties lie in finding conditions on $G$ to make this equality holds and make the above ``argument" actually work for all $x\in X$. We prove that $G$ being (non-torsion) one-ended is the key condition to get the necessary condition (\ref{E-2}). This is done in Corollary \ref{L-coincidence of plus and minus Livsic's test cocycles}, Lemmas \ref{L-plus and minus cocycles equal imply the extension of cocycle}, and \ref{L-specification and equal of cocycles imply directional triviality of cocycles}. 

In the second step, we will look for a good element $g\in G$ such that from (\ref{E-2}) we will get a solution $b$ of the equation (\ref{E-1}) in the direction $g$. To do that, we introduce another crucial terminology of our work: specification property for shift actions. Specification property which was introduced by Rufus Bowen in 1970s has played an important role in studying hyperbolic dynamical systems. After that, various specification properties have been introduced by many other authors \cite{CL,CT,CT1,LS,Pavlov,Thompson}. For the cocycle superrigidity phenomena of $\Z^d$-shifts, Schmidt introduced another modified version of specification property \cite{Schmidt95}. His definition depends heavily on the Euclidean structure of $\Z^d$ and hardly to extend to general groups because of the lack of inner product. Using geometric group theory idea, we can overcome this difficulty to define a totally new version of specification property for group actions. And it is powerful enough for us to finish the task in this second step.

And in the final step, we show that when $G$ has one end, we can get a common solution $b$ for all non-torsion element $``g"$ we used in the second step, and in fact this transfer map $b$ also works for the whole group $G$. It follows from Lemmas \ref{cocycle triviality passes to normalizer} and \ref{trivial cocycles in two directions generates a subgroup}.



The paper is organized as following: in Section 2, we study continuous cocycle of subshifts via homoclinic equivalence relation as Schmidt did when the acting group is $\Z^d$ \cite{Schmidt95}. In Section 3, we introduced two new versions of specification properties for subshifts over general groups. These properties play crucial roles in our paper. In Section 4, we combine preparation results in Sections 2 and 3 to give a proof for Theorem \ref{T-one end is sufficient}, and we provide the generalized golden mean subshifts as a rich source for Theorem \ref{T-one end is sufficient}.
In Section 5, we present proofs of Theorem \ref{one end is necessary} and Corollary \ref{main theorem} by using a cohomological characterization of numbers of ends of groups in \cite{Dunwoody, Houghton, Specker}. Note that this characterization of ends of groups recently was also used to study in a different topic: ergodicity of principal algebraic actions  \cite{LiPetersonSchmidt}. Finally, we combine our results with \cite[Theorem 1.6]{Lixin} to  obtain continuous orbit equivalence rigidity of some actions of certain torsion-free amenable groups.

Some of results and techniques also hold when the target group has a compatible bi-invariant metric, meaning that they also work for groups in $\cU_{fin}$. For simplicity, in this article, we only deal with subshifts and discrete target groups. 

\section{Continuous cocyle, ends of groups, and homoclinic equivalence relation}\label{section on terminology}

\subsection{Cocycle rigidity}

Let $G$ be a countable group and act on a compact topological space $X$ by homeomorphisms and we denote this dynamical system by $(G,X)$. Let $H$ be a topological group.
\begin{definition}
A continuous $H$-cocycle for the system $(G,X)$ is a map $c:G\times X\to H$ such that $c(g,\cdot):X\to H$ is continuous for every $g\in G$ and $$c(g_1g_2,x)=c(g_1,g_2x)c(g_2,x), $$
for all $g_1,g_2\in G, x\in X$.
\end{definition}
If $\rho:G\to H$ is a homomorphism then clearly the map $G\times X \to H, (g,x)\mapsto \rho(g)$ which we still denote by $\rho$ is a continuous cocycle.

Let $b:X\to H$ be a continuous map and $c:G\times X\to H$ be a continuous $H$-cocycle then the map $c':G\times X\to H$ defined by $c'(g,x):=b(gx)^{-1}c(g,x)b(x)$ is also a continuous $H$-cocycle.

\begin{definition}
Two continuous $H$-cocycles $c,c':G\times X\to H$ are cohomologous if there exists a continuous map $b:X\to H$ called a transfer map such that $c(g,x)=b(gx)^{-1}c'(g,x)b(x)$ for every $x\in X,g\in G$.

We say a continuous $H$-cocycle $c: G\times X\to H$ is trivial if it is cohomologous to a homomorphism from $G$ to $H$.
\end{definition}

\begin{definition}\label{definition of cocycle-rigid}
A system $(G,X)$ is called continuous $H$-cocycle rigid if every continuous $H$-cocycle $c:G\times X\to H$ is cohomologous to a homomorphism $\rho:G\to H$. We say the system $(G,X)$ is superrigid if $(G,X)$ is continuous $H$-cocycle rigid for every discrete group $H$.
\end{definition}

A metric $d$ on a topological group $H$ is \textit{bi-invariant} if $d(gx,gy)=d(x,y)=d(xg,yg)$, for every $x,y,g\in H$. Note that if $d$ is a bi-invariant metric on $H$ then $d(x_1\cdots x_n,y_1\cdots y_n)\leq \sum_{i=1}^nd(x_i,y_i)$, for every $n\in\N, x_1,\cdots, x_n, y_1, \cdots, y_n\in H$.

\begin{definition}
The dynamical system $(G,X)$ is \textit{topological mixing} if $G$ is infinite and for every two non-empty open subsets $U, V$ of $X$, there exists a finite subset $F$ of $G$ such that $gU\cap V\neq \varnothing $ for every $g\in G\setminus F$. It is \textit{topological mixing of order $r$} if $G$ is infinite and for all nonempty open subsets $U_1, \cdots, U_r$ of $X$, there exists a finite subset $F$ of $G$ such that $\cap_{1\leq i\leq r}g_iU_i\neq \varnothing $ whenever $g_1, \cdots, g_r$ are elements of $G$ such that $g_i^{-1}g_j\not\in F$ for all $1\leq i<j\leq r$.
\end{definition}

The measurable version of the following lemma is well-known, see e.g. \cite[Lemma 3.5]{Furman},  \cite[Lemma 3.6]{Popa2}. Its proof is inspired by \cite[Theorem 3.2]{Schmidt95}.

\begin{lem}
\label{cocycle triviality passes to normalizer}
Let $G$ be a countable group and $\alpha$ be an action of $G$ on a compact metrizable space $X$ by homeomorphisms. Let $H$ be a topological group with a bi-invariant metric $d$ on $H$. Let $c:G\times X\to H$ be a continuous $H$-cocycle. Assume that there exists a subgroup $G_0<G$ and a continuous map $b:X\to H$ such that for any fixed $g$ in $G_0$ the map $x\mapsto b(gx)^{-1}c(g,x)b(x)$ is constant on $X$, and $\alpha|_{G_0}$ is topological mixing. Then $c: N_G(G_0)\times X\to H$ is cohomologous to a homomorphism, where $N_G(G_0)$ is the normalizer of $G_0$ in $G$.
\end{lem}

\begin{proof}
Let $c':G\times X\to H$ be a continuous $H$-cocycle defined by $c'(g,x):=b(gx)^{-1}c(g, x)b(x)$, for every $g\in G, x\in X$. Fix $g_0\in G_0$, let $L(g_0)$ be the common value of $b(g_0x)^{-1}c(g_0,x)b(x)$ for all $x\in X$.

First we claim for any fixed $g\in N_G(G_0)$, $c'(g, \cdot)$ is constant on $X$.

Fix $g\in N_G(G_0)$, and we denote by $g_0'=gg_0g^{-1}\in G_0$.
\begin{eqnarray*}
L^{-1}(g_0')c'(g, g_0x)L(g_0)&=&b(gx)^{-1}c^{-1}(g_0',gx)b(g_0'gx)b(gg_0x)^{-1}c(g, g_0x)b(g_0x)b(g_0x)^{-1}c(g_0, x)b(x)\\
&=&b(gx)^{-1}c^{-1}(g_0',gx)b(gg_0x)b(gg_0x)^{-1}c(g, g_0x)c(g_0, x)b(x)\\
&=&b(gx)^{-1}c^{-1}(g_0',gx)c(g, g_0x)c(g_0, x)b(x)\\
&=&b(gx)^{-1}c^{-1}(g_0',gx)c(gg_0, x)b(x)\\
&=&b(gx)^{-1}c^{-1}(g_0',gx)c(g_0'g, x)b(x)\\
&=&b(gx)^{-1}c(g, x)b(x)\\
&=&c'(g,x).
\end{eqnarray*}
And hence for every $g\in N_G(G_0),x\in X, g_0\in G_0$, one has
\begin{align}
\label{F-1}
c'(g,x)=L^{-1}(gg_0g^{-1})c'(g, g_0x)L(g_0).
\end{align}
Let $g\in N_G(G_0)$ and put $\varphi(\cdot)=c'(g,\cdot)$. Suppose that $\varphi$ is not constant, then there exist $\varepsilon>0, h_1,h_2\in H$, and non-empty open subsets $U_1, U_2$ of $X$ such that $\sup_{x\in U_i}d(\varphi(x),h_i)<\varepsilon$, for every $i=1,2$, and $d(h_1,h_2)>4\varepsilon$.  Since $\alpha|_{G_0}$ is topologically mixing, there exists a finite subset $F$ of $G_0$ such that $g_0U_1\cap U_i\neq \varnothing$ for every $g_0\in G_0\setminus F$ and $i=1,2$. Therefore, there exist $y_i\in g_0U_1\cap U_i$, so $d(\varphi(g_0^{-1}y_i),h_1)<\varepsilon$ and $d(\varphi(y_i), h_i)<\varepsilon$, for every $i=1, 2$. And hence from (\ref{F-1}), one has $d(L^{-1}(gg_0g^{-1})\varphi(y_i)L(g_0),h_1)<\varepsilon$ for every $i=1, 2$. Thus, since $d$ is bi-invariant we have
\begin{eqnarray*}
d(h_1,h_2)&\leq& d(h_2,\varphi(y_2))+d(\varphi(y_2),\varphi(y_1))+d(\varphi(y_1),h_1)\\
&<&2\varepsilon+d(L^{-1}(gg_0g^{-1})\varphi(y_2)L(g_0),L^{-1}(gg_0g^{-1})\varphi(y_1)L(g_0))\\
&\leq &2\varepsilon+d(L^{-1}(gg_0g^{-1})\varphi(y_2)L(g_0),h_1)+d(h_1,L^{-1}(gg_0g^{-1})\varphi(y_1)L(g_0))\\
&<&4\varepsilon \mbox { (contradiction).}
\end{eqnarray*}
Thus $\varphi$ is constant, and we may write $L(g)=c'(g,\cdot)$ when $g\in N_G(G_0)$.

Now it suffices to prove that $L$ is a homomorphism on $N_G(G_0)$.

For any $g_1, g_2$ in $N_G(G_0)$, one has
\begin{eqnarray*}
L(g_1g_2)&=&b(g_1g_2x)^{-1}c(g_1g_2, x)b(x)\\
&=&b(g_1g_2x)^{-1}c(g_1, g_2x)b(g_2x)b(g_2x)^{-1}c(g_2, x)b(x)\\
&=&L(g_1)L(g_2)\\
L(g_1^{-1})&=& b(g_1^{-1}x)^{-1}c(g_1^{-1}, x)b(x)\\
&=&[b(g_1g_1^{-1}x)^{-1}c(g_1, g_1^{-1}x)b(g_1^{-1}x)]^{-1}\\
&=&L(g_1)^{-1}
\end{eqnarray*}
\end{proof}
\begin{remark}
\begin{enumerate}
\item If $H$ is discrete, then the discrete metric $d$ defined on $H$ by $d(x,y)=1$ if $x\neq y$ and 0 otherwise is a bi-invariant metric.
\item If $H$ is a compact metrizable group then there exists a bi-invariant metric on $H$ \cite[Lemma C.2]{EW}.
\end{enumerate}
\end{remark}

\subsection{Ends of groups}
We recall that a map $\varphi:X\to Y$ between topological spaces is called to be proper if $\varphi^{-1}(K)$ is compact for every compact subset $K$ of $Y$. A ray in a topological space $X$ is a map $\varphi:[0,\infty)\to X$. Two continuous proper rays $\varphi_1, \varphi_2:[0,\infty)\to X$ are said to have the same end if for every compact subset $K$ of $X$, there exists $N\in\N$ such that $\varphi_1([N,\infty))$ and $\varphi_2([N,\infty))$ are contained in the same path components of $X\setminus K$. This defines an equivalence relation on the set of continuous proper rays of $X$ and we denote by $End(X)$ the set of all such equivalence classes. If $|End(X)|=k$, we say that $X$ has $k$ ends.
\begin{definition}
Let $G$ be a finitely generated group with $S$ is a finitely generating set. Let $X$ be its Cayley graph with respect to $S$. We define the set of ends of the group $G$, $End(G):=End(X)$.
\end{definition}
It is known that $End(G)$ does not depend on the choice of the finite generating set $S$ and every group $G$ has either $0,1,2$, or infinitely many ends.  Morever, $G$ has 0 ends if and only if $G$ is finite, 2 ends if and only if it has a subgroup of finite index isomorphic to $\Z$, and according to Stallings' Ends Theorem  \cite{Stallings}, $G$ has infinitely many ends if and only if $G = A*_CB$
or $G = A*_\varphi$ with $|A/C|\geq 3$, $|B/C|\geq 2$ and $\varphi$ is an
isomorphism between finite subgroups having index $\geq 2$ in A. These properties mentioned above may be found in \cite[Section 13.5]{Geoghegan}, \cite[Section IV.25.vi]{Harpe} or \cite[Theorem I.8.32]{BH}, and reference therein. 

Now we present several examples of groups with one end, the class of groups that we use in this paper.
\begin{example}\label{E-groups with one end}
\begin{enumerate}
\item Every amenable group that is not virtually cyclic has one end (applying Stalling's theorem \cite{Stallings}, or see \cite{MV} for a short proof).\\
Before discussing of the next example, let us review the definition of $1$-cocycle. Let $G$ be a countably infinite group and $M$ a left $\Z G$-module. An $1$-cocyle $c:G\to M$ is a map satisfying $$c(gh)=c(g)+gc(h), $$ for every $g,h\in G$. An $1$-cocycle $c:G\to M$ is a coboundary if there exists an $m\in M$ such that $c(g)=m-gm$ for every $g\in G$. We denote by $Z^1(G,M)$ and $B^1(G,M)$ the space of $1$-cocycles and coboundaries, respectively. The quotient space $Z^1(G,M)/B^1(G,M)$ is called the first cohomology of the group $G$ with coefficients in $M$ and is denoted by $H^1(G,M)$.
\item If a nonamenable group $G$ has zero first $\ell^2$-Betti number, then it has one end.
\begin{proof}
  From \cite[Corollary 2.4]{PetersonThom}, we know that $H^1(G, \ell^2(G))=0$. Then applying the Remark after Proposition 1 in \cite{BV}, one has $H^1(G, \mathbb{Z}G)=0$, which is equivalent to that $G$ has one end \cite[Theorem 4.6]{Houghton} (see also \cite[Theorem 13.5.5]{Geoghegan}). 
\end{proof}
\end{enumerate}

\end{example}

\subsection{Homoclinic equivalence relation in shifts}

Let $A$ be a finite set. The space $A^G:=\{x:G\to A\}$ endowed with the product topology is a compact metrizable space. The group $G$ acts on $A^G$ by the left shifts defined by $(gx)_h:=x_{g^{-1}h}$, for every $g,h\in G,x\in A^G$. A \textit{subshift} is a closed $G$-invariant subset of $A^G$. Let $X\subset A^G$ be a subshift. For a nonempty finite subset $F$ of $G$ and $x\in A^G$,
 we denote by $x_F$ the restriction of $x$ to $F$, and define $X_F:=\{x_F:x\in X\}$. If $F=\{g\}$ for some $g\in G$, we write it simply as $x_g$. Let $R\subset X\times X$ be an equivalence relation on $X$. For every $x\in X$, we denote by $R(x)$ the equivalence class of $x$ consisting of all $y\in X$ such that $(x,y)\in R$. A map $c:R\to H$ is a cocycle if $$c(x_1,x_2)c(x_2,x_3)=c(x_1,x_3) \mbox{  ,} \forall (x_1,x_2),(x_2,x_3)\in R.$$ We consider the \textit{homoclinic equivalence relation} $\Delta_X\subset X\times X$ defined by $$\Delta_X:=\{(x,x')\in X\times X: \mbox{ there exists a finite subset } F \mbox{ of } G \mbox{ such that } x_g=x'_{g},\forall g\in G\setminus F\}.$$
For every $x\in X$, we denote by $\Delta_X(x)$ the homoclinic equivalence class of $x$ in $X$.
From now on, $G$ is always a finitely generated group. Let $S$ be a finite symmetric generating set of $G$. Let $\ell_S$ and $d$ be the word length and word metric on $G$ induced by $S$, respectively. For every $r\geq 0$, we put $B(r):=\{g\in G:\ell_S(g)\leq r\}$. Assume that there is an element $g\in G$ such that $\ell_S(g^k)\to \infty$ as $k\to\infty$, equivalently, $g$ has infinite order in $G$. Let $H$ be a discrete group and $f:X\to H$ be a continuous map. Then one can define cocycles $c_f^{(g),\pm}:\Delta_X\to H$ as follows: for every $(x,x')\in\Delta_X$, and $m\geq 1$, we define
\begin{align*}
c_f^{(g),+,(m)}(x,x') & =\Big(\prod_{k=0}^{m-1}f(g^kx)^{-1}
\Big)\Big(\prod_{k=0}^{m-1}f(g^kx')^{-1}\Big)^{-1},\\
c_f^{(g),-,(m)}(x,x') & =\Big(\prod_{k=1}^{m-1}f(g^{-k}x)
\Big)\Big(\prod_{k=1}^{m-1}f(g^{-k}x')\Big)^{-1}.
\end{align*}
For every $(x,x')\in\Delta_X$, there exists a finite subset $F$ of $G$ such that $x_h=x'_h$ for every $h\in G\setminus F$. Because $H$ is discrete and $f$ is continuous, there exists a non-empty finite subset $F_1$ of $G$ such that $f(y)=f(z)$ for every $y,z\in X$ with $y_{F_1}=z_{F_1}$. Because $\ell_S(g^k)\to\infty$ as $k\to\infty$, there exists $N\in \N$ such that $g^{-k}\not\in FF_1^{-1}$, for every $k\geq N$, i.e. $g^{-k}F_1\subset G\setminus F$, for every $k\geq N$. Thus for every $k\geq N$, one has $(g^kx)_{F_1}=x_{g^{-k}F_1}=x'_{g^{-k}F_1}=(g^kx')_{F_1}$, and hence
\begin{eqnarray*}
c_f^{(g),+,(m)}(x,x')  &=&\Big(\prod_{k=0}^{m-1}f(g^kx)^{-1}
\Big)\Big(\prod_{k=0}^{m-1}f(g^kx')^{-1}\Big)^{-1}\\
&=&\Big(\prod_{k=0}^{N-1}f(g^kx)^{-1}
\Big)\Big(\prod_{k=0}^{N-1}f(g^kx')^{-1}\Big)^{-1}=c_f^{(g),+,(N)}(x,x'),
\end{eqnarray*}
for every $m\geq N$. Then we can define a map $c_f^{(g),+}:\Delta_X\to H$ by $$c_f^{(g),+}(x,x'):=\lim_{m\to \infty}c_f^{(g),+,(m)}(x,x') \mbox{   (*)  }, $$ for every $(x,x')\in \Delta_X$ and $c_f^{(g),+}$ will satisfy the cocycle condition $$c_f^{(g),+}(x_1,x_2)c_f^{(g),+}(x_2,x_3)=c_f^{(g),+}(x_1,x_3) \mbox{  ,}\forall (x_1,x_2),(x_2,x_3)\in \Delta_X.$$ Similarly, we also have a cocycle map $c_f^{(g),-}:\Delta_X\to H$ defined by $$c_f^{(g),-}(x,x'):=\lim_{m\to \infty}c_f^{(g),-,(m)}(x,x') \mbox{  (**)  } ,$$ for every $(x,x')\in \Delta_X$.

\begin{lem}\label{test to show same transfer map for different directions}
Let $G$ be a group with a finite symmetric generating set $S$. Assume $G$ has one end and $g, h$ are elements in $G$ with infinite order, then $c_{f_g}^{(g), +}(x, y)=c_{f_h}^{(h), +}(x, y)$ for all $(x, y)\in\Delta_X$, where $f_g(x):=c(g, x), ~f_h(x):=c(h, x)$ for all $x\in X$ and $c: G\times X\to H$ is a continuous cocycle with $(G, X)$ being a subshift and $H$ being a discrete group.
\end{lem}

\begin{proof}

From the definition of $f_g$ and $c_{f_g}^{(g), +}$, it is easy to see that 
\begin{eqnarray*}
c_{f_g}^{(g), +}(x, y)&=&\lim_{n\to\infty}c(g^n, x)^{-1}c(g^n, y),\\
c_{f_h}^{(h), +}(x, y)&=&\lim_{n\to\infty}c(h^n, x)^{-1}c(h^n, y).
\end{eqnarray*}

Since the target group is discrete, actually both limits are obtained for sufficiently large $n$, say for all $n>N_1$.

Now, since $(x, y)\in\Delta_X$, we may assume $x_g=y_g$ for all $g\not\in B(r)$ for some $r>0$.

Since $S$ is a finite set, the target group is discrete and $c$ is a continuous cocycle, we may find a finite set $F\subset G$ such that $c(s, x)=c(s, y)$ for all $s\in S$ and all $x, y\in X\subset A^G$ whenever $x_F=y_F$.

Now, pick $\ell>0$ large enough such that $FB(r)\subset B(\ell)$.

Then, since $\lim_{n\to\pm\infty}\ell_S(g^n)=\infty, \lim_{n\to\pm\infty}\ell_S(h^n)=\infty$, we may find $N_0$ such that $g^n, h^n\not\in B(\ell+\ell_S(g)+\ell_S(h)+2)$ for all $|n|>N_0$.

Fix $N> N_0, N_1$.

Observe that there is an $\infty$-path $P_g$ passing through $\{g^k: k\leq -N\}$. Indeed, write $g^{-1}=t_1t_2\cdots t_{\ell_S(g)}$, where $t_i\in S$ for all $1\leq i\leq \ell_S(g)$, then we can take $P_g$ to be the one which starts with $g^{-N}$, then passes through $$g^{-N}t_1, g^{-N}t_1t_2, \cdots, g^{-N}t_1t_2\cdots t_{\ell_S(g)-1}, g^{-N-1}, g^{-N-1}t_1, \cdots.$$  

Note that $P_g$ lies outside of $B(\ell)$ by the choice of $N_0$.

Similarly, we also have an $\infty$-path $P_h$ passing through $\{h^k: k\leq -N\}$ which lies outside of $B(\ell)$.

Now, since $G$ has one end, both $g^{-N}$ and $h^{-N}$ lie in the unique $\infty$-connected component of $G\setminus B(\ell)$. This implies there exists a path in $G\setminus B(\ell)$ connecting $g^{-N}$ and $h^{-N}$, i.e. there exists a finite sequence $s_i, i=1, \cdots, k$ with all $s_i\in S$ such that $h^{-N}{s_1}^{-1}{s_2}^{-1}\cdots {s_i}^{-1}\not\in B(\ell)$, for all $i=1,\cdots, k$ and $g^{-N}=h^{-N}s_1^{-1}\cdots s_k^{-1}$; since $S$ is symmetric, this is equivalent to say that $s_i\cdots s_2s_1h^N\not\in B(\ell)$, for all $i=1, \cdots, k$ and $g^N=s_k\cdots s_1h^N$. 

Then, the cocycle identity implies the following:
\begin{eqnarray*}
c(g^N, x)&=&c(s_k\cdots s_1h^N, x)\\
&=&c(s_k, s_{k-1}\cdots s_1h^Nx)c(s_{k-1}, s_{k-2}\cdots s_1h^Nx)\cdots c(s_1, h^Nx)c(h^N, x).\\
\end{eqnarray*}
Now, we claim that $c(s_i, s_{i-1}\cdots s_1h^Nx)=c(s_i, s_{i-1}\cdots s_1h^Ny)$ for all $i=1,\cdots, k$. 

To check this, we only need to show $(s_{i-1}\cdots s_1h^Nx)_F=(s_{i-1}\cdots s_1h^Ny)_F$, or equivalently, $x_{(s_{i-1}\cdots s_1h^N)^{-1}F}=y_{(s_{i-1}\cdots s_1h^N)^{-1}F}$ for all $i=1,\cdots, k$. It suffices to show that $(s_{i-1}\cdots s_1h^N)^{-1}F\cap B(r)=\emptyset$, i.e. $(s_{i-1}\cdots s_1h^N)\not\in FB(r)^{-1}=FB(r)$. This is true since $s_{i-1}\cdots s_1h^N\not\in B(\ell)$ and $FB(r)\subset B(\ell)$.

Hence, we deduce that 
\begin{eqnarray*}
&&c(s_k, s_{k-1}\cdots s_1h^Nx)c(s_{k-1}, s_{k-2}\cdots s_1h^Nx)\cdots c(s_1, h^Nx)\\
&=&c(s_k, s_{k-1}\cdots s_1h^Ny)c(s_{k-1}, s_{k-2}\cdots s_1h^Ny)\cdots c(s_1, h^Ny).
\end{eqnarray*}

Therefore, we have 

\begin{eqnarray*}
c_{f_g}^{(g), +}(x, y)&=&c(g^N, x)^{-1}c(g^N, y)\\
&=&c(h^{N}, x)^{-1}c(h^{N}, y)\\
&=&c_{f_h}^{(h), +}(x, y).
\end{eqnarray*}

\end{proof}

\begin{cor}\label{L-coincidence of plus and minus Livsic's test cocycles}
Let $G$ be a group with a finite symmetric generating set $S$. Assume $G$ has one end and $g$ is an element in $G$ with infinite order, then $c_{f_g}^{(g), +}(x, y)=c_{f_g}^{(g), -}(x, y)$ for all $(x, y)\in\Delta_X$, where $f_g(x):=c(g, x)$ for all $x\in X$ and $c: G\times X\to H$ is a continuous cocycle with $(G, X)$ being a subshift and $H$ being a discrete group.
\end{cor}
\begin{proof}
From definition of $c_{f_g}^{(g), -}$, we deduce that 

$c_{f_g}^{(g), -}(x, y)=\lim_{n\to\infty}c(g^{-n}, x)^{-1}c(g^{-n}, y)=c_{f_{g^{-1}}}^{(g^{-1}), +}(x, y)$. 

So we may apply the above lemma by setting $h=g^{-1}$. 
\end{proof}

\begin{remark}
(1) The special case $G=\Z^d, ~d>1$ of Corollary \ref{L-coincidence of plus and minus Livsic's test cocycles} is proved implicitly in the proof of \cite[Theorem 3.2]{Schmidt95}, but by a direct calculation using the structure of $\Z^d$. 

(2) We could not drop the assumption that $G$ has one end in Corollary \ref{L-coincidence of plus and minus Livsic's test cocycles}. Indeed, one can check that the cocycle constructed in Example \ref{Z is not in G} fails this lemma, and in fact $\Z\not\in\mathcal{G}$, see \cite[Introduction]{katokschmidt} for more discussion on this. 

(3) We can slightly relax the assumption in Lemma \ref{test to show same transfer map for different directions}. Indeed, the proof still works if we only assume there exists a finitely generated, one-ended subgroup $G_1<G$ such that the non-torsion elements $g, h\in G_1$. So the assumption that $G$ has one end is not used so crucially here as in Lemma \ref{trivial cocycles in two directions generates a subgroup}.
\end{remark}

\section{Specification properties of shifts}
Specification is a strong orbit tracing property. It was introduced first by Rufus Bowen in the early of 1970s to investigate the uniqueness of equilibrium state of certain maps. Ruelle \cite{Ruelle} extended the notion to $\Z^d$ -actions. Several definitions of various specification
properties were introduced in \cite[Definition 5.1]{LS} (see also \cite[Remark 5.6]{LS}) to study algebraic actions of $\Z^d$. These have been modified to
the general group case in \cite[Section 6]{CL}. Recently, a numerous weaker forms of specification property for $\Z$-shifts have been introduced to study intrinsic ergodicity and irregular sets \cite{CT,CT1,Pavlov, Thompson}. To deal with cocycle superrigidity for $\Z^d$-shifts, Schmidt suggested a modified specification property which depends heavily on the Euclidean structure of $\Z^d$ \cite{Schmidt95}. In this section, we give two more new different versions of specification properties or shifts over general groups which play important roles in our work. 

\begin{definition}\label{definition of specification}
Let $G$ be a finitely generated group with a finite generating set $S$. Let $a$ be an element in $G$. For every $r\geq 0$, we define
$$\Delta^+(a,r):=\{(x,x')\in \Delta_X: x_{\{a^{k}: k\geq 0\}B(r)}=x'_{\{a^{k}: k\geq 0\}B(r)}\},$$
$$\Delta^-(a,r):=\{(x,x')\in \Delta_X: x_{\{a^{-k}:k\geq 0\}B(r)}=x'_{\{a^{-k}:k\geq 0\}B(r)}\}.$$
We say that the equivalence class $\Delta_X(\bar{x})$ of a point $\bar{x}\in X$ has strong $a$-specification (respectively, $a$-specification) if $\Delta_X(\bar{x})\cap a^{-1}\Delta_X(\bar{x})$(respectively, $\Delta_X(\bar{x})$) is dense in $X$, and if the following property holds: for any $r\geq 0$ there exists some finite integer $N=N(a, r)\geq r$, such that if $x,x'\in \Delta_X(\bar{x})$ satisfy that $x_{B(N(a, r))}=x'_{B(N(a, r))}$ then we can find an element $y$ in $\Delta_X(\bar{x})$ such that $(x,y)\in \Delta^+(a,r)$ and $(x',y)\in \Delta^-(a,r)$. And we say that the homoclinic equivalence relation $\Delta_X$ has strong $a$-specification (respectively, $a$-specification) if there exists a point $\bar{x}\in X$ such that $\Delta_X(\bar{x})$ has strong $a$-specification (respectively, $a$-specification).
\end{definition}
\begin{remark}
Although we do not put assumption on the element $a$, in application, we are mainly interested in $a$ being non-torsion. 
\end{remark}
The following observation will make sense for the existence of the point $y$ in our definition of specification.
\begin{lem}\label{L-specification of group}
Let $a$ be an element of $G$ with infinite order, then for every $r\geq 0$, there exists some finite integer $N=N(a, r)\geq r$, such that ${\{a^{k}: k\geq 0\}B(r)}\cap {\{a^{k}: k\leq 0\}B(r)}\subset B(N(a, r))$.
\end{lem}
\begin{proof}
If $z=a^kx=a^{-\ell}y$, where $x, y\in B(r)$ and $k, \ell\geq 0$, then $a^{k+\ell}=yx^{-1}\in B(2r)$. Since $a$ is not torsion, $\lim_{n\to\infty}\ell_S(a^n)=\infty$. Therefore, $S_r:=\{n\geq 0:~a^n\in B(2r)\}$ is a finite set, and we denote the largest element in $S_r$ by $M=M(a, r)$, then we deduce $k\leq k+\ell\leq M$ since $k+\ell\in S_r$.

Now, $\ell_S(z)=\ell_S(a^kx)\leq \ell_S(a^k)+r\leq \sum_{i=0}^M\ell_S(a^i)+r=:N$. 

So, ${\{a^{k}: k\geq 0\}B(r)}\cap {\{a^{k}: k\leq 0\}B(r)}\subset B(N(a, r))$.

\end{proof}
\begin{lem}\label{full shift has specification}
Let $(G, A^G)$ be the full shift. If $a\in G$ is not torsion, then $\Delta_X(\bar{x})$ has strong $a$-specification, where $\bar{x}=(0)_G$, i.e. $\bar{x}$ is the element in $X=A^G$ with every coordinate to be a constant $0\in A$.
\end{lem}

\begin{proof}

For any $z\in A^G$ and any finite set $F\subset G$, if we define $y\in A^G$ by setting $y_g=z_g$ if $g\in F$ and $y_g=\bar{x}_g$ if $g\not\in F$, then since $\bar{x}$ is $G$-invariant, $y\in \Delta_X(\bar{x})\cap a^{-1}\Delta_X(\bar{x})$. Hence $\Delta_X(\bar{x})\cap a^{-1}\Delta_X(\bar{x})$ is dense in $X$.

Now, suppose $x,x'\in \Delta_X(\bar{x})$ and they satisfy that $x_{B(N(a, r))}=x'_{B(N(a, r))}$, then we can define $y\in \Delta_X(\bar{x})$ by setting $y_g=x_g$ if $g\in {\{a^{k}: k\geq 0\}B(r)}$, $y_g=x'_g$ if $g\in {\{a^{k}: k\leq 0\}B(r)}$ and $y_g=\bar{x}_g$ elsewhere. Note that Lemma \ref{L-specification of group} guarantees this $y$ is well-defined.
\end{proof}

\begin{lem}
\label{L-plus and minus cocycles equal imply the extension of cocycle}
Let $X\subset A^G$ be a subshift and $f: X\to H$ be a continuous map, where $H$ is a discrete group. Assume that $\Delta_X(\bar{x})$ has $g$-specification and the cocycles $c_f^{(g),\pm}:\Delta_X\to H$ in (*) and (**) are equal.
Then $\lim_{\substack{(x,x')\to\Delta \\ x,x'\in\Delta_X(\bar{x})}}d(c_f^{(g),+}(x,x'),e_H)=0$, where $\Delta\subset X\times X$ denotes the diagonal.
\end{lem}
\begin{proof}
Since $H$ is discrete and $f$ is continuous, there exists a non-empty finite subset $F$ of $G$ such that $f(y)=f(z)$ for every $y,z\in X$ with $y_F=z_F$. Choose $r_0\in\N$ such that $F\subset B(r_0)$. Then for every $r\geq r_0,k\geq 0, (x,x')\in \Delta^+(g,r)$, one has $(g^{-k}x)_F=x_{g^{k}F}= x'_{g^{k}F}=(g^{-k}x')_F$ and hence $f(g^{-k}x)=f(g^{-k}x')$. Then since $d$ is bi-invariant, for every $r\geq r_0$, $(x,x')\in \Delta^+(g,r)$, one has
 \begin{eqnarray*}
 d(c_f^{(g),-}(x,x'),e_H)&=&\lim_{m\to \infty}d\Big(\Big(\prod_{k=1}^{m-1}f(g^{-k}x)
\Big)\Big(\prod_{k=1}^{m-1}f(g^{-k}x')\Big)^{-1},e_H\Big)\\
&= & \lim_{m\to \infty}d\Big(\prod_{k=1}^{m-1}f(g^{-k}x),
\prod_{k=1}^{m-1}f(g^{-k}x')\Big)\\
&\leq &\lim_{m\to \infty}\sum_{k=1}^{m-1}d(f(g^{-k}x),f(g^{-k}x'))=0.
 \end{eqnarray*}
Since $\Delta_X(\bar{x})$ has $g$-specification, for every $r\geq 0$, there exists $N=N(g, r)\geq r$ such that for all $x,x'\in \Delta_X(\bar{x})$ with $x_{B(N(g, r))}=x'_{B(N(g, r))}$ there exists an element $y\in \Delta_X(\bar{x})$ with $(x,y)\in \Delta^+(g,r)$ and $(x',y)\in \Delta^-(g,r)$. As before, whenever $r\geq r_0$ we know that $f(g^kx')=f(g^ky)$ for every $k\geq 0$ and then
\begin{eqnarray*}
&&\lim_{m\to \infty}d\Big(\Big(\prod_{k=0}^{m-1}f(g^kx)^{-1}
\Big)\Big(\prod_{k=0}^{m-1}f(g^kx')^{-1}\Big)^{-1},\Big(\prod_{k=0}^{m-1}f(g^kx)^{-1}
\Big)\Big(\prod_{k=0}^{m-1}f(g^ky)^{-1}\Big)^{-1}\Big)\\
&=&\lim_{m\to \infty}d\Big(\Big(\prod_{k=m-1}^{0}f(g^kx')
\Big),\Big(\prod_{k=m-1}^{0}f(g^ky)
\Big)\Big)=0.
\end{eqnarray*}
Hence, for every $x,x'\in \Delta_X(\bar{x})$ with $x_{B(N(g, r))}=x'_{B(N(g, r))}, r\geq r_0$, one has $$d(c_f^{(g),+}(x,x'), c_f^{(g),+}(x,y))=0$$ and thus
\begin{eqnarray*}
d(c_f^{(g),+}(x,x'),e_H)&\leq &d(c_f^{(g),+}(x,x'),c_f^{(g),+}(x,y))+d(c_f^{(g),+}(x,y),e_H)\\
&=&d(c_f^{(g),+}(x,x'),c_f^{(g),+}(x,y))+d(c_f^{(g),-}(x,y),e_H)=0.
\end{eqnarray*}
Taking $r\to \infty$, we finish the proof.
\end{proof}

\section{One end is sufficient}
In this section, we will prove Theorem \ref{T-one end is sufficient}. To prove it we will prepare several lemmas.

Let $G$ be a finitely generated group with $S$ being a finite symmetric generating set and $g\in G$ be an element with infinite order. Let $X\subset A^G$ be a subshift for some finite set $A$ and $H$ be a discrete group with discrete metric $d$. Let $f:X\to H$ be a continuous function.
\begin{lem}
\label{L-specification and equal of cocycles imply directional triviality of cocycles}
 Assume that there exists a point $\bar{x}\in X$ such that $\Delta_X(\bar{x})\cap g^{-1}\Delta_X(\bar{x})$ is dense in $X$ and $\lim_{\substack{(x,x')\to\Delta \\ x,x'\in\Delta_X(\bar{x})}}d(c_f^{(g),+}(x,x'),e_H)=0$, then there is a continuous map $b:X\to H$ such that the map $x\mapsto b(gx)^{-1}f(x)b(x)$ is constant on $X$.
\end{lem}

\begin{proof}
Since $\lim_{\substack{(x,x')\to\Delta \\ x,x'\in\Delta_X(\bar{x})}}d(c_f^{(g),+}(x,x'),e_H)=0$, $c_f^{(g),+}$ is a cocycle and $\Delta_X(\bar{x})$ is dense in $X$, one can define a function $\hat{c}:X\times \Delta_X(\bar{x})\to H$ by $\hat{c}(x,y):=\lim_{n\to\infty}c_f^{(g),+}(x_n,y)$, for some (and every) sequence $\{x_n\}\subset \Delta_X(\bar{x})$ with $x_n\to x$. Then we define the map $b:X\to G$ by $b(x):=\hat{c}(x,\bar{x})$ for every $x\in X$. 

We claim the function $b$ is continuous.

Given $\epsilon>0$ and $x\in X$, first take $\delta>0$ such that if $d(z, w)<\delta$ and $z, w\in\Delta_X(\bar{x})$, then $d(c_f^{(g),+}(z, w), e_H)<\frac{\epsilon}{3}$, this is possible by the assumption. Now, for any $y\in X$ with $d(x, y)<\frac{\delta}{2}$, and any fixed $x_n\to x, y_n\to y$ with $x_n, y_n\in\Delta_X(\bar{x})$ for all $n\geq 1$, we take $n$ large enough such that $d(x_n, y_n)<\delta$, $d(b(x), b(x_n))<\frac{\epsilon}{3}$ and $d(b(y), b(y_n))<\frac{\epsilon}{3}$ (by the definition of $b$), note that then $d(b(x_n), b(y_n))=d(c_f^{(g),+}(x_n, y_n), e_H)<\frac{\epsilon}{3}$ by the choice of $\delta$, now we get

$d(b(x), b(y))\leq d(b(x), b(x_n))+d(b(x_n), b(y_n))+d(b(y_n), b(y))<\epsilon$.

One also has $b(x)b(x')^{-1}=\hat{c}(x,\bar{x})\hat{c}(x',\bar{x})^{-1}=c_f^{(g),+}(x,\bar{x})c_f^{(g),+}(x',\bar{x})^{-1}=c_f^{(g),+}(x,x')$ for all $(x,x')\in \Delta_X(\bar{x})\times \Delta_X(\bar{x})$.

We mention that (although we do not need this fact in the proof) we can further extend $\hat{c}$ to a function (which we still denote by) $\hat{c}: X\times X\to H$ by setting $\hat{c}(x, x'):=\lim_{x_n\to x, y_n\to x', x_n, y_n\in\Delta_X(\bar{x})}c_f^{(g),+}(x_n, y_n)$. Now we check it is well-defined. Take any $x_n, y_n, z_n, w_n\in \Delta_X(\bar{x})$ for all $n\geq 1$ and $x_n\to x, z_n\to x, y_n\to y, w_n\to y$, we have the following:
\begin{eqnarray*}
&&d(c_f^{(g),+}(x_n, y_n), c_f^{(g),+}(z_n, w_n))\\
&=&d(c_f^{(g),+}(x_n, z_n)c_f^{(g),+}(z_n, w_n)c_f^{(g),+}(w_n, y_n), c_f^{(g),+}(z_n, w_n))\\
&\leq& d(c_f^{(g),+}(x_n, z_n)c_f^{(g),+}(z_n, w_n)c_f^{(g),+}(w_n, y_n), c_f^{(g),+}(z_n, w_n)c_f^{(g),+}(w_n, y_n))\\&+& d(c_f^{(g),+}(z_n, w_n)c_f^{(g),+}(w_n, y_n), c_f^{(g),+}(z_n, w_n))\\
&=&d(c_f^{(g),+}(x_n, z_n), e_H)+d(c_f^{(g),+}(w_n, y_n), e_H)\to 0.\\
\end{eqnarray*}
Then it is clear that $b(x)b(x')^{-1}=\hat{c}(x, x')$ for all $(x, x')\in X\times X$, hence $\hat{c}$ is a continuous function.
%

Let $x, x'\in\Delta_X(\bar{x})\cap g^{-1}\Delta_X(\bar{x})$, then we deduce
\begin{eqnarray*}
b(x)b(x')^{-1}&=& c_f^{(g), +}(x, x')\\
&=&\lim_{m\to\infty}(\prod_{k=0}^mf(g^kx)^{-1})(\prod_{k=0}^mf(g^kx')^{-1})^{-1}\\
&=&\lim_{m\to\infty}f(x)^{-1}(\prod_{k=1}^mf(g^kx)^{-1})(\prod_{k=1}^mf(g^kx')^{-1})^{-1}f(x')\\
&=&f(x)^{-1}c_f^{(g), +}(gx, gx')f(x')\\
&=&f(x)^{-1}b(gx)b(gx')^{-1}f(x').\\
\end{eqnarray*}
Since $\Delta_X(\bar{x})\cap g^{-1}\Delta_X(\bar{x})$ is dense in $X$ and $b,f$ are continuous functions, the above holds for all $x$ in $X$, we finish the proof.

\end{proof}
\begin{remark}
In the above proof, the assumption $\lim_{\substack{(x,x')\to\Delta \\ x,x'\in\Delta_X(\bar{x})}}d(c_f^{(g),+}(x,x'),e_H)=0$ guarantees that $c_f^{(g), +}$ is continuous on $\Delta_X(\bar{x})\times \Delta_X(\bar{x})$, but may not necessarily on the bigger space $\Delta_X\subset X\times X$, although $c_f^{(g), +}$ is well-defined on $\Delta_X$; however, the cocycle identity ensures that starting from $c_f^{(g), +}|_{\Delta_X(\bar{x})\times \Delta_X(\bar{x})}$, the extension $\hat{c}$ is continuous on $X\times X$. It is not clear whether $\hat{c}$ is equal to $c_f^{(g), +}$ on $\Delta_X\setminus (\Delta_X(\bar{x})\times \Delta_X(\bar{x}))$. 
\end{remark}

We have the following lemma.

\begin{lem}\label{trivial cocycles in two directions generates a subgroup}
Assume $G$ has one end, $g_i\in G$ is of infinite order, and there exists a point $\bar{x}\in X$ such that $\Delta_X(\bar{x})\cap {g_i}^{-1}\Delta_X(\bar{x})$ is dense in $X$ and $$\lim_{\substack{(x,x')\to\Delta \\ x,x'\in\Delta_X(\bar{x})}}d(c_{f_{g_i}}^{(g_i),+}(x,x'),e_H)=0,$$ for all $i=1,\cdots, k$, where $k$ is any positive integer (maybe infinity), then there is a continuous map $b: X\to H$ such that the map $x\mapsto b(gx)^{-1}c(g, x)b(x)$ is constant on $X$ (depending only on $g$) for all $g\in \langle g_1, \cdots, g_k \rangle$, where $c: G\times X\to H$ is a continuous cocycle and $H$ is discrete.
\end{lem}

\begin{proof}
From the proof of Lemma \ref{L-specification and equal of cocycles imply directional triviality of cocycles}, we know that for $i=1,\cdots, k$, the continuous map $b_i: X\to H$ defined by $b_i(x):=\lim_{n\to\infty}c_{f_{g_i}}^{(g_i),+}(x_n, \bar{x})$ where $x_n\in\Delta_X(\bar{x})$ and $x_n\to x$ will make the map $x\mapsto b_i(g_ix)^{-1}c(g_i, x)b_i(x)$ constant (only depending on $g_i$). Now we claim that $b_i(x)=b_j(x)$ for all $x\in X$ and all $1\leq i, j\leq k$.

By Lemma \ref{test to show same transfer map for different directions}, we know for any $n$, $c_{f_{g_i}}^{(g_i),+}(x_n, \bar{x})=c_{f_{g_j}}^{(g_j),+}(x_n, \bar{x})$. Hence, $b_i(x)=b_j(x)$. Then, we can define $b(x):=b_1(x)$.

Now, from the identities $b(g^{-1}x)^{-1}c(g^{-1}, x)b(x)=[b(gg^{-1}x)^{-1}c(g, g^{-1}x)b(g^{-1}x)]^{-1}$ and $b(ghx)^{-1}c(gh, x)b(x)=[b(ghx)^{-1}c(g, hx)b(hx)][b(hx)^{-1}c(h, x)b(x)]$, it is clear that the map $x\mapsto b(gx)^{-1}c(g, x)b(x)$ is constant on $X$ (depending only on $g$) for all $g\in \langle g_1, \cdots, g_k \rangle$.
\end{proof}

\begin{lem}\label{torsion-free elements form a normal subgroup}
Let $G$ be an infinite group that is not torsion and $G_{\infty}$ be the subgroup generated by all non-torsion elements in $G$, 
then $G_{\infty}$ is an infinite normal subgroup of $G$.
\end{lem}
\begin{proof}
This is obvious since $order(ghg^{-1})=order(h)$.
\end{proof}

\begin{proof}[Proof of Theorem \ref{T-one end is sufficient}]

We apply successively Corollary \ref{L-coincidence of plus and minus Livsic's test cocycles}, Lemma \ref{L-plus and minus cocycles equal imply the extension of cocycle}, Lemma \ref{L-specification and equal of cocycles imply directional triviality of cocycles}, Lemma \ref{trivial cocycles in two directions generates a subgroup} to conclude that the map $x\mapsto b(gx)^{-1}c(g, x)b(x)$ is constant on $X$ (depending only on $g$) for all $g\in G_{\infty}$, where $G_{\infty}$ is the subgroup of $G$ generated by all non-torsion elements. Since $G_{\infty}$ is normal in $G$ by Lemma \ref{torsion-free elements form a normal subgroup}, then we apply Lemma  \ref{cocycle triviality passes to normalizer} to conclude that $c$ is a trivial cocycle on the whole group $G$.
\end{proof}
\subsection{Examples}
In this subsection, we provide more examples that are continuous superrigid actions besides full shifts. 

Given a dynamical system $(G,X)$, a point $x\in X$ is called \textit{periodic} if its orbit $\{gx:g\in G\}$ is finite. For $x\in X$, we denote by $stab(x):=\{g\in G:gx=x\}$, the stabilizer of $x$. It is clear that $stab(x)$ is a subgroup of $G$ for every $x\in X$, and $x$ is a periodic point if and only if $[G:stab(x)]<\infty$. Let $H$ be a subgroup with finite index, we define $X_H:=\{x\in X:hx=x, \mbox{ for all } h\in H\}$, the $H$-periodic points. Then the set of periodic points $Per(X)$ will be $\bigcup_{H<G, [G:H]<\infty}X_H$.

A subshift $X$ of $A^G$ is of \textit{finite type} if there is a finite set $F$ of $G$ and a subset $P$ of $A^F$ such that $$X=\{x\in A^G:(gx)_F\in P, \forall g\in G\}.$$
The subset $F$ and $P$ is called a \textit{defining window} and a \textit{set of allowed words} for $X$.
It is clear that if a shift of finite type is defined by a defining window $F$ and a set of allowed words $P\subset A^F$, then for every $F'\supset F$, it is also a shift of finite type defined by $F'$ and $P':=\{x\in A^{F'}:x_F\in P\}$.

The following lemma is an extension of \cite[Lemma 2.5]{Schmidt95} from $\Z^d$-actions to actions of more general groups. The proof of the first part here is inspired by the proof of \cite[Theorem 1.1]{CC2012}.
\begin{lem}
\label{L-density of periodic points and homoclinic points}
Let $X\subset A^G$ be a subshift of finite type. Suppose that there exists a periodic point $\bar{x}$ such that $\Delta_X(\bar{x})$ is dense in $X$.
\begin{enumerate}
\item
If $G$ is a residually finite group, then $\Per(X)$ is dense in $X$.
\item If $G$ is a countably infinite group, then $(G, X)$ is topological mixing of every order.
\end{enumerate}
\end{lem}
\begin{proof}
Let $D$ be a finite subset of $G$ such that $e_G\in D$, $D$ is a defining window and $P$ is a set of allowed words $P\subset A^D$ for $X$.

(1) Let $x\in X$ and $\Omega$ be a finite subset of $G$. It suffices to show that there exists a periodic point $y$ in $X$ such that $x_\Omega=y_\Omega$. Since $\Delta_X(\bar{x})$ is dense in $X$, there exists $z\in \Delta_X(\bar{x})$ such that $z_\Omega=x_\Omega$. Let $F$ be a finite subset of $G$ such that $z_{G\setminus F}=\bar{x}_{G\setminus F}$. Put $F_1=FD^{-1}$ and $F_2=(F_1\cup\Omega)D^{-1}D$. Since $G$ is residually finite, we can find a normal subgroup $H$ of $G$ with finite index such that $g_1H\cap g_2H=\varnothing$ for every $g_1\neq g_2\in F_2$. As $stab(\bar{x})$ is a subgroup of finite index in $G$, applying Lemma \cite[Lemma 2.1.10]{CC2010}, there exists a normal subgroup $K$ of finite index in $G$ such that $stab(\bar{x})\supset K$. Then $L:=H\cap K$ is a normal subgroup of finite index in $G$. Let $y$ be an element in $A^G$ defined by \[
  y_g = \left\{
  \begin{array}{@{}c@{\quad}l@{}}
    z_{g_1} & \text{if $gL=g_1L$ for some $g_1\in F_2$}\\
    \bar{x}_g & \text{otherwise}\\
  \end{array}\right.
\]
Since $g_1H\cap g_2H=\varnothing$ for every $g_1\neq g_2\in F_2$, one has  $g_1L\cap g_2L=\varnothing$ for every $g_1\neq g_2\in F_2$, and hence $y$ is well defined.
We will prove that $y$ is a periodic point we are looking for. First, one has $y_\Omega=z_\Omega=x_\Omega$ since
 $\Omega\subset F_2$, and $y$ is a periodic point because $y$ is constant on every right coset of $L$. Now we only need to prove $y\in X$. It suffices to show that for every $g\in G$, there exists $u^{(g)}\in X$ such that $y_{gD}=u^{(g)}_{gD}$.

Now we consider the first case, if for every $g_1\in F_1$ and $g_2\in gD$, $g_1L\neq g_2L$ then we will show that $y_{g_2}=\bar{x}_{g_2},$ for every $g_2\in gD$. If $g_2L\neq hL$, for every $h\in F_2$, then by the definition of $y$, it is done. If there exists $k\in F_2\setminus F_1$ such that $g_2L=kL$ then $y_{g_2}=z_{k}=\bar{x}_k=\bar{x}_{g_2}$ since $z_{G\setminus F}=\bar{x}_{G\setminus F}$, $F\subset F_1$, and $\bar{x}$ is an $L$-periodic point.

Suppose now there exist $g_1\in F_1$ and $g_2=g\delta\in gD$ such that $g_1L=g_2L$. Since $L$ is normal, one has $Lg_1=Lg_2$. Then there exists $h\in L$ such that $h^{-1}g\delta=g_1$. We will show that $hz$ is the point we are looking for. For all $l=g\delta'\in gD$, one has $h^{-1}l=h^{-1}g\delta\delta^{-1}\delta'=g_1\delta^{-1}\delta'\in F_1D^{-1}D\subset F_2$, and hence
$y_l=(hy)_l=y_{h^{-1}l}=z_{h^{-1}l}=(hz)_l$.

(2) We may take a normal subgroup $H$ of $G$ with some finite index $n$ such that $H\subset stab(\bar{x})$, then we decompose $G$ as $\bigsqcup_{i=1}^nHz_i$ for finitely many $z_i\in G$ with $z_1=e_G$.

Fix $r\geq 2$ and nonempty, both closed and open sets $\cO_1, \cdots, \cO_r\subset X$, we claim that there exists a finite subset $F$ of $G$ such that $\cap_{1\leq i\leq r}g_i\cO_i\neq \varnothing $ whenever $g_1, \cdots, g_r$ are elements of $G$ and $g_i^{-1}g_j\not\in F$ for $1\leq i<j\leq r$.

Let $W=\{z_i\cO_j: 1\leq j\leq r ~\text{and}~ 1\leq i\leq n  \}$.

Since $\Delta_X(\bar{x})$ is dense in $X$, for any nonempty closed and open set $\cO\subset X$, there exist a finite set $F_{\cO}\subset G$ and a point $x^{\cO}\in \cO\cap \Delta_X(\bar{x})$ such that $x^{\cO}_{G\setminus F_{\cO} }=\bar{x}_{G\setminus F_{\cO} }$ and for any $y\in X$, if $y_{F_{\cO}}=x^{\cO}_{F_{\cO}}$, then $y\in \cO$.

Note that we may enlarge $F_{\cO}$ while preserving the above property, so without loss of generality, we may assume for every choice of $\cO$ in $W$, all $F_{{\cO}}$ are the same, and we denote it by $F_1$.

Define $F=\{z_i^{-1}F_1F_1^{-1}z_j: 1\leq i, j\leq n\} \cup \{z_i^{-1}F_1D^{-1}D F_1^{-1}z_j: 1\leq i, j\leq n\}$, and fix $g_1, \cdots, g_r\in G$ with $g_i^{-1}g_j\not\in F$ for $1\leq i<j\leq r$, we write them as $g_i=h_iz_{\sigma(i)}$ for some map $\sigma: \{1,\cdots, r\}\to \{1,\cdots, n\}$ and $h_i$ in $H$ for all $1\leq i\leq r$.

For every choice of $\cO'_j\in W$, $j=1, \cdots, r$, we can define a (unique) point $y\in A^G$ by setting $(h_j^{-1}y)_{F_1}=(x^{\cO'_j})_{F_1}$ for all $1\leq j\leq r$ and $y_{G\setminus \cup_{j=1}^rh_jF_1}=\bar{x}_{G\setminus \cup_{j=1}^rh_jF_1}$.

Now we claim the following:

(i) $y$ is well defined;

(ii) $y\in X$ (hence $y\in\Delta_X(\bar{x})$ by definition);

(iii) $y\in \cap_{1\leq j\leq r}h_j\cO_j'$.

To check (i), it is enough to show $h_iF_1\cap h_jF_1=\varnothing$ for all $1\leq i<j\leq r$. Assume not, then $h_i^{-1}h_j\in F_1F_1^{-1}$, while $h_i^{-1}h_j=z_{\sigma(i)}g_i^{-1}g_jz_{\sigma(j)}^{-1}$, so $g_i^{-1}g_j\in {z_{\sigma(i)}}^{-1}F_1F_1^{-1}z_{\sigma(j)}\subset F$, which is a contradiction.

To check (ii), we need to show that $y_{gD}\in P$ for all $g\in G$. Fix $g\in G$, we first show that $gD$ can have nonempty intersection with at most one of $h_iF_1$. Suppose not, then $gD\cap h_iF_1\neq\varnothing$ and $gD\cap h_jF_1\neq\varnothing$ for some $i<j$. Then, $g\in h_iF_1D^{-1}\cap h_jF_1D^{-1}$, hence $h_iF_1D^{-1}\cap h_jF_1D^{-1}\neq\varnothing$ and then $h_i^{-1}h_j\in F_1D^{-1}D F_1^{-1}$, so $g_i^{-1}g_j\in {z_{\sigma(i)}}^{-1}F_1D^{-1}D F_1^{-1}z_{\sigma(j)}\subset F$, which is a contradiction. Now we can break our argument in two cases.

Case 1, $gD\subset G\setminus \cup_{j=1}^rh_jF_1$. Then, $y_{gD}=\bar{x}_{gD}\in P$.

Case 2, $gD\cap h_jF_1\neq\varnothing$ for some (and unique) $j$. Then $y_{gD}=y_{(gD\cap h_jF_1)\bigsqcup (gD\cap G\setminus h_jF_1)}$.
If $h_js\in gD\cap h_jF_1$, then $y_{h_js}=(x^{\cO'_j})_s$. If $t\in gD \cap (G\setminus h_jF_1)$, then $h_j^{-1}t\not\in F_1$, hence $y_t=\bar{x}_t=(h_j\bar{x})_t=\bar{x}_{h_j^{-1}t}=(x^{\cO'_j})_{h_j^{-1}t}$. Then, $y_{gD}=(h_jx^{\cO'_j})_{gD}\in P$.

To check (iii), it suffices to check $(h_j^{-1}y)_{F_1}=(x^{\cO'_j})_{F_1}$ for all $j$, which is clear by the definition of $y$.
\end{proof}

A dynamical system $(G,X)$ is called \textit{surjunctive} if every continuous $G$-commuting map $\varphi:X\rightarrow X$ being injective is automatically surjective.
\begin{cor}
Let $G$ be a countably residually finite group and let $G$ acts on a compact metrizable space $X$ by homeomorphisms. Assume that the system $(G,X)$ is expansive and there exists an extension $(G,Y)$ of $(G,X)$ such that $Y\subset A^G$ is a subshift of finite type for some finite set $A$ and $Y$ contains a periodic point $\bar{y}$ such that $\Delta_X(\bar{y})$ is dense in $Y$. Then the system $(G,X)$ is surjunctive.
\end{cor}
\begin{proof}
Let $\varphi:Y\to X$ be an extension map. Then $\varphi(Per(Y))\subset Per(X)$ because $\varphi$ is $G$-equivariant. From Lemma \ref{L-density of periodic points and homoclinic points}, we know that $Per(Y)$ is dense in Y. And as $\varphi$ is continuous and surjective, one has $Per(X)$ is dense in $X$ and hence applying \cite[Proposition 5.1]{CC2015}, one has $(G,X)$ is surjunctive.
\end{proof}
\begin{definition}
Let $\Omega$ be a non-empty finite subset of $G$. A subshift $X\subset A^G$ is said to have $\Omega$-propagation if for every non-empty finite subset $F$ of $G$ and $x\in A^F$ satisfying $x_{F\cap g\Omega}\in X_{F\cap g\Omega}$ for every $g\in G$ then one has $x\in X_F$. And $X$ is said to have bounded propagation if $X$ has $\Omega$-propagation for some non-empty finite subset $\Omega$ of $G$.
\end{definition}
The definition of bounded propagation for subshifts was introduced by Gromov \cite[page 160]{Gromov99} and was studied extensively in \cite{Fiorenzi,CC2012}.
\begin{example}(\textbf{The generalized golden mean subshifts}) Let $A=\{0,1,\cdots, k\}$ and $F_1,\cdots, F_m$ be non-empty finite subsets of $G$. Let $X(F_1,\cdots, F_m)$ be a subset of $A^G$ consisting of all $x\in A^G$ such that for every $1\leq j\leq m$, $g\in G$, there exists $g_j\in gF_j$ such that $x_{g_j}=0$. Then it is a bounded propagation subshift of $A^G$ \cite[Example 2.8]{CC2012}.
\begin{lem}
\label{L-specification of generalized golden mean shift}
Let $G$ be a finitely generated group with a finite generating set $S$ and $a$ be an element in $G$ with infinite order. Let $A=\{0,1,\cdots, k\}$ and $F_1,\cdots, F_m$ be non-empty finite subsets of $G$. We denote by $X$ the generalized golden mean subshifts with respect to $F_1,\cdots, F_m$. Then $\Delta_X$ has strong $a$-specification.
\end{lem}
\begin{proof}
For every $r\geq 0$, we put $P^+(a,r):=\{a^{k}: k\geq 0\}B(r)$ and $P^-(a,r):=\{a^{-k}: k\geq 0\}B(r)$. Let $\bar{x}$ be the constant $0$ configuration of $X$ then it is clear that $\Delta_X(\bar{x})\cap a^{-1}\Delta_X(\bar{x})$ is dense in $X$. Choose $r_0\in\N$ such that $F_j\subset B(r_0)$ for every $1\leq j\leq m$. Let $r>0$. By Lemma \ref{L-specification of group}, there exists $M\geq r$ such that $$P^+{(a,r+r_0)}\cap P^-{(a,r+r_0)}\subset B(M).$$
Write $N=M+r_0$, we will prove that for every $x,x'\in \Delta_X(\bar{x})$ such that $x_{B(N)}=x'_{B(N)}$, there exists $y\in\Delta_X(\bar{x})$ such that $(x,y)\in \Delta^+(a,r)$ and $(x',y)\in \Delta^-(a,r)$. Let $x,x'\in \Delta_X(\bar{x})$ such that $x_{B(N)}=x'_{B(N)}$, we define $y:G\to A$ by
\[
  y_g = \left\{
  \begin{array}{@{}c@{\quad}l@{}}
    x_g  & \text{if } g\in P^+{(a,r)}\\
    x'_{g} & \text{if } g\in P^-{(a,r)}\\
   0 & \text{otherwise}\\
  \end{array}\right.
\]
Then $y$ has non-zero values at only finitely many places. Let $g\in G$, $1\leq j\leq m$. It suffices to show that there exists $g_j\in gF_j$ such that $y_{g_j}=0$. If $gF_j\cap (P^+{(a,r)}\cup P^-{(a,r)})=\varnothing$, i.e. $gF_j\subset G\setminus (P^+{(a,r)}\cup P^-{(a,r)})$ then $y_{gF_j}=\{0\}$. Now we consider the case $gF_j\cap (P^+{(a,r)}\cup P^-{(a,r)})\neq\varnothing$. If $gF_j\subset P^+{(a,r)}$ or $gF_j\subset P^-{(a,r)}$ then $y_{gF_j}=x_{gF_j}$ or $y_{gF_j}=x'_{gF_j}$ respectively, and hence we are done. If there exist $f_1, f_2\in F_j$ such that $gf_1\in P^+{(a,r)}$ and $gf_2\in P^-{(a,r)}$, then by assumption on the element $a$, one has $$g\in P^+{(a,r)}f_1^{-1}\cap P^-{(a,r)}f_2^{-1}\subset P^+{(a,r+r_0)}\cap P^-{(a,r+r_0)}\subset B(M).$$
And hence $gF_j\subset B(M+r_0)=B(N)$. As $x_{B(N)}=x'_{B(N)}$, and from the definition of $y$, $y_{g_j}=0$ for some $g_j\in gF_j$.
\end{proof}

\end{example}
\begin{cor}
Let $A=\{0,1,\cdots, k\}$, $G$ be a finitely generated group satisfying the conditions as in Theorem \ref{T-one end is sufficient}, and $F_1,\cdots, F_m$ be non-empty finite subsets of $G$. Let $X\subset A^G$ be the generalized golden mean subshift with respect to $F_1,\cdots, F_m$. Then $(G, X)$ is superrigid.
\end{cor}
\begin{proof}
Let $\bar{x}$ be the constant $0$ configuration of $X$ then $\bar{x}$ is a periodic point. Applying \cite[Proposition 2.9]{CC2012}, we know that $X$ is a subshift of finite type. Then combining Lemmas \ref{L-density of periodic points and homoclinic points}, \ref{L-specification of generalized golden mean shift}, and Theorem \ref{T-one end is sufficient}, we finish the proof.
\end{proof}
\begin{remark}
Let $G$ be a finitely generated group and $S=\{s_1,\cdots,s_m\}$ be a generating set of $G$. Let $F_j=\{e_G,s_j\}$ and $X(S)\subset \{0,1,\cdots, k\}^G$ the generalized golden mean subshift with respect to $F_1,\cdots, F_m$. Then $X(S)$ consists of all configurations in which the $i$'s are isolated for every $1\leq i\leq k$, i.e. $X(S)$ consists of all elements $x\in \{0,1,\cdots, k\}^G$ such that for every $g\in G$, $x_{gs_j}=0$  whenever $x_g\neq 0$.
\end{remark}

\section{One end is necessary}

In this section, we prove that if $G\in\mathcal{G}$, then $G$ has one end. Let us first observe that the free groups $F_r(r\geq 1)$ do not belong to $\mathcal{G}$. The result should be well-known, but we can not find it in literature. Thus, we present a proof here.

\begin{example}\label{Z is not in G}
Let $G=F_1=\Z$, we consider the continuous cocycle $c: \Z\times (\frac{\Z}{2\Z})^{\Z}\to \frac{\Z}{2\Z}$ defined by setting $c(g, x)=x_g$, where $g$ is the generator of $\Z$. Observe that we have two $x$ (i.e. the one which is constant 1 and the one which is constant 0) such that $gx=x$. If there exist $b: (\frac{\Z}{2\Z})^{\Z}\to \frac{\Z}{2\Z}$ and a group homomorphism $\phi: \Z\to \frac{\Z}{2\Z}$ such that $c(g, x)=-b(gx)+\phi(g)+b(x)$ for all $x\in (\frac{\Z}{2\Z})^{\Z}$, then we apply this equality to such two $x$, we get $\phi(g)=0$ and $\phi(g)=1$, a contradiction.

Similar construction also applies to $G=F_r$, the free groups on $r$ generators, say $s_1, \cdots, s_r$. Indeed, we may consider the cocycle $c: F_r\times (\frac{\Z}{2\Z})^{F_r}\to \frac{\Z}{2\Z}$ defined by setting $c(s_i, x)=x_{s_i}$ for all $x\in (\frac{\Z}{2\Z})^{F_r}$ and all $1\leq i\leq r$. Note that this cocycle is well-defined since there are no relations among these generators. It is not hard to conclude that this cocycle is not trivial by a similar argument as above.
\end{example}

Note that the above construction does not apply to $G=\Z^2$, because if we define $c: \Z^2\times (\frac{\Z}{2\Z})^{\Z^2}\to \frac{\Z}{2\Z}$ by setting $c(a, x)=x_a, c(b, x)=x_b$, where $\Z^2=\langle a, b \rangle$, then $ab=ba$ implies $c(a, bx)+c(b, x)=c(ab, x)=c(ba, x)=c(b, ax)+c(a, x)$ for all $x$. So $x_{b^{-1}a}+x_b=x_{a^{-1}b}+x_a$, but not every $x$ in $(\frac{\Z}{2\Z})^{\Z^2}$ satisfy this relation, so this $c$ is not well-defined on the whole space.

Now let us prove if $G\in\mathcal{G}$, then $G$ has one end. This might be compared with \cite[Corollary 4.2]{petersonsinclair}. 

\begin{proof}[Proof of Theorem \ref{one end is necessary}]

From Specker's characterization for ends of groups via the associated first cohomology groups \cite{Specker} (see its proofs in \cite[Theorem 1]{Dunwoody} or \cite[Theorem 4.6]{Houghton}), we know that if $G$ has more than one end, then $H^1(G, \oplus_GC_2)$ is nontrivial, i.e. there exists a 1-cocycle $c: G\to\oplus_GC_2$ which is not a 1-coboundary, where $C_2=\{\bar{0}, \bar{1}\}$ denotes the cyclic group with two elements under addition. Now, using such $c$, we could define a 1-cocycle $c': G\times (\frac{\Z}{2\Z})^G\to \{\pm 1\}\subset \mathbb{T}$ by setting $c'(g, x)=x(c(g^{-1}))$ for all $g\in G, x\in X=(\frac{\Z}{2\Z})^G$. 

Here, we treat $X$ as a dual group, i.e. $(X, +)=\widehat{\oplus_G{C_2}}$, and we treat $\frac{\Z}{2\Z}=\{\pm 1\}\subset \mathbb{T}$ as a group under multiplication. If $x=(x_g), y=(y_g)\in X$ and $a=\oplus_{g\in F}\bar{1}\in \oplus_G{C_2}$ for some finite set $F\subset G$, then $x(a):=\prod_{g\in F}x_g$ if $F\neq \emptyset$, and $x(a):=1$ if $F=\emptyset$. $x+y\in X$ is defined by $(x+y)_g:=x_gy_g$ for all $g\in G$.  

It is clear that $c'$ is continuous. We claim that $c'$ is not trivial as a $\{\pm 1\}\cong\frac{\Z}{2\Z}$-valued cocycle.

Assume not, then there exist a group homomorphism $\rho: G\to\frac{\Z}{2\Z}$ and a continuous map $b: X\to \frac{\Z}{2\Z}$ such that $c'(g, x)=b(gx)^{-1}\rho(g)b(x)$ for all $g\in G$ and all $x\in X$. Now, we apply the above equation to $x'=(1)_G\in X$, where $1\in\frac{\Z}{2\Z}=\{\pm 1\}$, and note that $gx'=x'$ for all $g\in G$, then we deduce $\rho(g)=1$ for all $g\in G$; hence, $x(c(g^{-1}))=c'(g, x)=b(gx)^{-1}b(x)$ for all $g\in G$ and all $x\in X$. This implies $c$ is a 1-coboundary, the proof is the same as the proof of \cite[Theorem 1.1]{Jiang}, we include it for completeness.

Since $X$ is a dual group, we have $x(c(g^{-1}))y(c(g^{-1}))=(x+y)(c(g^{-1}))$ for all $x, y\in X$, then $b(gx)^{-1}b(x)b(gy)^{-1}b(y)=b(gx+gy)^{-1}b(x+y)$ for all $g\in G$ and all $x, y\in X$, equivalently, $b(x)b(y)b(x+y)^{-1}=b(gx)b(gy)b(gx+gy)^{-1}$ for all $g\in G$ and all $x\in X$.

Now, using the fact that $G\curvearrowright (X\times X, \mu\times \mu)$ is ergodic, where $\mu$ is the Haar measure on $X$, the above implies $b(x)b(y)b(x+y)^{-1}=\lambda$ for some constant $\lambda\in \mathbb{T}$ and $\mu\times \mu$-a.e. $(x, y)\in X\times X$.

Then after replacing $b$ with $\lambda^{-1} b$, which does not change the original equation on $c'$ and $b$, and applying \cite[Lemma 2.7]{Jiang}, we may assume $b\in\widehat{X}=\oplus_{G}C_2$. Then we deduce $x(c(g^{-1}))=x(-g^{-1}b+b)$ for all $x\in X$, hence $c(g)=b-gb$ for all $g\in G$
by Pontryagin duality, i.e. $c: G\to\oplus_GC_2$ is a 1-coboundary, which is a contradiction.

\end{proof}
\begin{remark}
We observe that for continuous cocycles, \cite[Theorem 1.1]{Jiang} still holds. In particular, it implies that if the shift $(G, \mathbb{T}^G)$ is continuous $\mathbb{T}$-cocycle rigid, then $G$ has one end. Indeed, if $G$ has more than one end, $H^1(G, \Z G)=H^1(G, \oplus_G\Z)$ is nontrivial, the same proof as above stills works here after replacing the target group with $\mathbb{T}$. In fact, the proof of Theorem \ref{one end is necessary} is obtained as we try to apply this observation. But since in Theorem \ref{one end is necessary}, the target group is discrete and the base space is a finite set instead of $\mathbb{T}$, we need to take advantage of \cite[Theorem 4.6]{Houghton}.
 \end{remark}
 \begin{proof}[Proof of Corollary \ref{main theorem}]
 Applying Theorems \ref{T-one end is sufficient}, \ref{one end is necessary}, and Lemma \ref{full shift has specification}, we get the corollary.
 \end{proof}

\section{Continuous Orbit Equivalence}\label{section on applications}
In this section we will present an application of Theorem \ref{T-one end is sufficient} to study rigidity of continuous orbit equivalence.
Let $G, H$ be countable groups and $X,Y$ be topological Hausdorff spaces.
\begin{definition}
The system $(G,X)$ is free if $gx\neq x$ for every $g\in G\setminus\{e_G\}, x\in X$; and topologically free if for every $e_G\neq g\in G$, one has $\{x\in X: gx\neq x\}$ is dense in $X$.
\end{definition}
\begin{definition}
Two topological dynamical systems $(G,X)$ and $(H,Y)$ are \textit{(continuous) orbit equivalent} which we denoted by $(G,X)\sim_{COE} (H,Y)$ if there exist continuous maps $b:G\times X \to H$, $c:H\times Y\to G$ and a homeomorphism $\varphi: X\to Y$ with the inverse $\psi:=\varphi^{-1}:Y\to X$ such that
\begin{align}
\varphi(gx)=b(g,x)\varphi(x),\\
\psi(hy)=c(h,y)\psi(y),
\end{align}
for every $x\in X,y\in Y,g\in G,h\in H$.
\end{definition}
\begin{remark}
\label{R-OE of topological free systems implies cocycle}
Note that if the system $(G,X)$ is topologically free then the map $c$ is uniquely determined by (3) and furthermore, it is a cocycle \cite[Remark 1.7 and Lemma 2.8]{Lixin}.
\end{remark}
\begin{definition}
Two topological dynamical systems $(G,X)$ and $(H,Y)$ are \textit{(topologically) conjugate} which we denoted by $(G,X)\sim_{conj} (H,Y)$ if there exist a group isomporphism $\rho:G \to H$ and a homeomorphism $\varphi: X\to Y$ such that $\varphi(gx)=\rho(g)\varphi(x)$, for every $g\in G,x\in X$.
\end{definition}
It is clear that if $(G,X)$ and $(H,Y)$ are topologically conjugate then they are orbit equivalent.
\begin{cor}\label{cocycle rigidity to coe}
Let $G$ be a finitely generated amenable and torsion-free group. Let $A$ be a finite set and $\sigma$ be the left shift action of $G$ on $A^G$ and $X\subset A^G$ a subshift such that the induced action of $\sigma$ on $X$ is topologically free. Assume that the system $(G,X)$ has properties as in Theorem \ref{T-one end is sufficient}. Then if $(G,X)$ is orbit equivalent to another topologically free system $(H,Y)$ then they are conjugate. In particular, if $H$ is a group being non-isomorphic to $G$ then $(G,X)$ is not orbit equivalent to any topologically free system $(H,Y)$.
\end{cor}
\begin{proof}
It follows from Theorem \ref{T-one end is sufficient}, Remark \ref{R-OE of topological free systems implies cocycle} and \cite[Theorem 1.6]{Lixin}.
\end{proof}

\textbf{Acknowledgement:} Part of this paper was carried out when Nhan-Phu Chung visited University of M\"{u}nster on May-July, 2015 and Chungnam National University on August-December, 2015. He is grateful to Wilhem Winter and Keonhee Lee for their warm hospitality at M\"{unster} and Daejeon, respectively. He thanks Xin Li for insightful conversations.  This work was partially supported by a grant from the Simons Foundation for Nhan-Phu Chung. 
We thank Hanfeng Li for many useful comments.

\end{document}